\documentclass{birkjour}
\usepackage{subfigure}
\usepackage{verbatim}
 \newtheorem{thm}{Theorem}[section]
 \newtheorem{corollary}[thm]{Corollary}
 \newtheorem{lemma}[thm]{Lemma}
 \newtheorem{Proposition}[thm]{Proposition}
 \theoremstyle{definition}
 
 \theoremstyle{remark}
 \newtheorem{remark}[thm]{Remark}
 
 \numberwithin{equation}{section}
 
 \newcommand{\R}{\mathbb{R}}
 \newcommand{\C}{\mathbb{C}}

\begin{document}

%
%
%
%
%
%
%
%
%

\title[ Equiaffine Characterization of Lagrangian Surfaces]
 {Equiaffine Characterization of Lagrangian \newline Surfaces in $\R^4$}

\author[M.Craizer]{Marcos Craizer}

\address{%
Catholic University \br
Rio de Janeiro\br
Brazil}

\email{craizer@puc-rio.br}

\thanks{The author want to thank CNPq for financial support during the preparation of this manuscript.}



\date{December 23, 2014}

\begin{abstract}
For non-degenerate surfaces in $\R^4$, a distinguished transversal bundle called affine normal plane bundle was proposed in \cite{Nomizu93}. 
Lagrangian surfaces have remarkable properties with respect to this normal bundle, like for example, the normal bundle being Lagrangian. 
In this paper we characterize those surfaces which are Lagrangian with respect to some parallel symplectic form in $\R^4$. 

\end{abstract}

\maketitle

\section{Introduction}

We consider non-degenerate surfaces $M^2\subset\R^4$. For such surfaces, there are many possible choices of the transversal bundle, 
and we consider here the {\it affine normal plane bundle} proposed in \cite{Nomizu93}. For affine mean curvature, umbilical surfaces and some other properties of this bundle we
refer to \cite{Dillen}, \cite{Magid}, \cite{Verstraelen} and \cite{Vrancken}. 
In this paper, considering the affine normal plane bundle, we give an equiaffine characterization of the Lagrangian surfaces. The results can be compared with 
\cite{Morvan87}, where a characterization of Lagrangian surfaces is given in terms of euclidean invariants of the surface. 

Consider the affine $4$-space $\R^4$ with the standard connection $D$ and a parallel volume form $[\cdot,\cdot,\cdot,\cdot]$.
Let $M\subset\R^4$ be a surface with a non-degenerate {\it Burstin-Mayer metric} $g$ (\cite{Burstin27}).
For a definite metric $g$, we write $\epsilon=1$, while for an indefinite metric $g$, we write $\epsilon=-1$. 
For a given transversal plane bundle $\sigma$ and $X,Y$ tangent vector fields,  write
\begin{equation}
D_XY=\nabla_XY+h(x,y),
\end{equation}
where $\nabla_XY$ is tangent to $M$ and $h(X,Y)\in\sigma$. Then $\nabla$ is a torsion free affine connection and $h$ is a symmetric bilinear form. 
For local vector fields $\{\xi_1,\xi_2\}$ defining a basis of $\sigma$, define the symmetric bilinear forms $h^1$ and $h^2$ by
\begin{equation}
h(X,Y)=h^1(X,Y)\xi_1+h^2(X,Y)\xi_2.
\end{equation}
Let $\{X_1,X_2\}$ be a local $g$-orthonormal tangent frame, i.e., $g(X_1,X_1)=\epsilon$, $g(X_1,X_2)=0$, $g(X_2,X_2)=1$.  For an arbitrary transversal plane bundle $\sigma$, it is proved in \cite{Nomizu93} that there exists a unique local basis $\{\xi_1,\xi_2\}$ 
of $\sigma$ such that
$[X_1,X_2,\xi_1,\xi_2]=1$ and
\begin{equation}\label{eq:Normalxi}
\begin{array}{c}
h^1(X_1,X_1)=1,\\
h^1(X_1,X_2)=0,\\
h^1(X_2,X_2)=-\epsilon,
\end{array}
\ \ \ \ 
\begin{array}{c}
h^2(X_1,X_1)=0,\\
h^2(X_1,X_2)=1,\\
h^2(X_2,X_2)=0.
\end{array}
\end{equation}
There are some transversal plane bundles $\sigma$ with distinguished properties, and we shall consider here the {\it affine normal plane bundle}
proposed in \cite{Nomizu93}. 

Assuming that $M$ is Lagrangian with respect to a parallel symplectic form $\Omega$, we shall verify the following remarkable facts:  (1) The affine normal plane bundle is $\Omega$-Lagrangian; 
(2) $\Omega\wedge\Omega=c[\cdot,\cdot,\cdot,\cdot]$, for some constant c; (3) $\Omega(X_1,\xi_2)-\Omega(X_2,\xi_1)=0$ and $\Omega(X_1,\xi_1)+\epsilon\Omega(X_2,\xi_2)=0$. Based on these 
facts, we shall describe the equiaffine conditions for a surface to be Lagrangian with respect to a parallel symplectic form.

Given a transversal bundle $\sigma$ and a local basis $\{\xi_1,\xi_2\}$, define the $1$-forms $\tau_i^j$, $i=1,2$, $j=1,2$, and the shape operators $S_i$ by
\begin{equation}\label{eq:shape}
D_X\xi_i=-S_iX+\tau_i^1(X)\xi_1+\tau_i^2(X)\xi_2,
\end{equation}
where $S_iX$ is in the tangent space. Writing
\begin{equation}\label{eq:definelambda}
S_iX_j=\lambda_{ij}^1X_1+\lambda_{ij}^2X_2,
\end{equation}
define 
\begin{equation}\label{eq:defineL}
\begin{array}{c}
L_{11}=\lambda_{11}^1-\lambda_{21}^2;\ \ L_{12}=-\epsilon\lambda_{11}^2-\lambda_{21}^1\\
L_{21}=\lambda_{12}^1-\lambda_{22}^2;\ \ L_{22}=-\epsilon\lambda_{12}^2-\lambda_{22}^1,
\end{array}
\end{equation}
and the $2\times 2$ matrix 
\begin{equation*}
L=\left[
\begin{array}{cc}
L_{11} & L_{12}\\
L_{21} & L_{22}
\end{array}
\right].
\end{equation*}
We shall verify that the rank of $L$ is independent of the choice of the $g$-orthonormal local frame $\{X_1,X_2\}$.

Consider the cubic forms $C^i$, $i=1,2$ given by
\begin{equation}\label{eq:defineCubic}
C^i(X,Y,Z)=\nabla_Xh^{i}(Y,Z)+\tau_1^{i}(X)h^1(Y,Z)+\tau_2^{i}(X)h^2(Y,Z).
\end{equation}
and define
\begin{equation}\label{eq:defineF}
\begin{array}{c}
F_{11}=3C^1(X_1,X_1,X_2)-\epsilon C^1(X_2,X_2,X_2)\\
F_{12}=\epsilon C^1(X_1,X_1,X_1)-3C^1(X_1,X_2,X_2)\\
F_{21}=3C^2(X_1,X_1,X_2)-\epsilon C^2(X_2,X_2,X_2)\\
F_{22}=\epsilon C^2(X_1,X_1,X_1)-3C^2(X_1,X_2,X_2).
\end{array}
\end{equation}
We shall verify that the rank of the matrix 
\begin{equation*}
F=\left[
\begin{array}{cc}
F_{11} & F_{12}\\
F_{21} & F_{22}
\end{array}
\right]
\end{equation*}
is also independent of the choice of the local $g$-orthonormal tangent frame $\{X_1,X_2\}$. In fact we shall prove that
the rank of the $2\times 4$ matrix 
$$
H=\left[\  L\  | \ F\ \right]
$$
is independent of the choice of the local frame. 

In case $rank(H)=1$, denote by $[A,B]^t$ a column-vector in the kernel of $H$ and let $\eta=\tan^{-1}(B/A)$, if $\epsilon=1$, and $\eta=\tanh^{-1}(B/A)$, if $\epsilon=-1$.
Define
\begin{equation}\label{eq:defineG}
\begin{array}{c}
G_1=\Gamma_{22}^2-\epsilon\Gamma_{11}^2+\tau_1^1(X_2)-\epsilon\tau_1^2(X_1)\\
G_2=\Gamma_{11}^1-\epsilon\Gamma_{22}^1-\tau_2^2(X_1)+\tau_1^2(X_2),
\end{array}
\end{equation}
where 
\begin{equation}\label{eq:defineChristoffel}
\nabla_{X_i}X_j=\Gamma_{ij}^1X_1+\Gamma_{ij}^2X_2. 
\end{equation}
We shall verify that, for the affine normal plane bundle, the conditions
\begin{equation}\label{eq:DerivEta}
\begin{array}{c}
d\eta(X_1)+\epsilon G_1=0\\
d\eta(X_2)-G_2=0,
\end{array}
\end{equation}
are independent of the choice of the local frame.

Our main theorem is the following:

\begin{thm}\label{thm1}
Given a surface $M\subset\R^4$, consider a local tangent frame $\{X_1,X_2\}$ and a local basis $\{\xi_1,\xi_2\}$ of the affine normal plane bundle $\sigma$ 
satisfying equations \eqref{eq:Normalxi}. 
\begin{enumerate}
\item Assume that there exists a parallel symplectic form $\Omega$ such that $L$ is $\Omega$-Lagrangian. Then the affine normal plane bundle is $\Omega$-Lagrangian, 
$\Omega\wedge\Omega=c[\cdot,\cdot,\cdot,\cdot]$, for some constant $c$ and $[A,B]^t$ belongs to the kernel of $H$, where $A=\Omega(X_1,\xi_2)=\Omega(X_2,\xi_1)$ and 
$B=\Omega(X_1,\xi_1)=-\epsilon\Omega(X_2,\xi_2)$. Moreover $\eta$ satisfies equations \eqref{eq:DerivEta}. 

\item If $rank(H)=1$ and $ker(H)$ satisfies equations \eqref{eq:DerivEta}, then there exists a parallel symplectic form $\Omega$ such that
$M$ is $\Omega$-Lagrangian. 
\end{enumerate}
\end{thm}

In order to complete the picture, it remains to consider what occurs under the hypothesis $H=0$. It is proved in \cite{Nomizu93} that, under the weaker hypothesis $F=0$, $M$ must be a complex curve,
if the metric $g$ is definite, or a product of planar curves, if $g$ is indefinite. In any case, it is well-known that there are two linearly independent parallel symplectic forms under which $M$ is Lagrangian
(see \cite{Craizer14}, \cite{Martinez05}). Thus we can write the following:

\begin{corollary}
A surface $M^2\subset\R^4$ is Lagrangian with respect to a parallel symplectic form if and only if $rank(H)=1$ and equations \eqref{eq:DerivEta} hold or $rank(H)=0$. 
\end{corollary}

The paper is organized as follows: In section 2 we describe the equiaffine invariants of a surface in $\R^4$, showing that $rank(H)$ is independent of the choice of the local frame. In section 3, we give a characterization of the affine normal bundle in terms of the cubic forms and show that equations \eqref{eq:DerivEta} are independent of the choice of the local frame. 
In section 4 we prove the main theorem.

\section{Shape Operators and Cubic Forms}

\subsection{The affine metric and local frames}

We begin by recalling the definition of the affine metric $g$ of a surface $M\subset\R^4$ (\cite{Burstin27},\cite{Nomizu93}). For a local frame $u=\{X_1,X_2\}$ of the tangent plane, let
$$
G_u(Y,Z)=\frac{1}{2}\left(  [X_1,X_2,D_YX_1,D_ZX_2]+    [X_1,X_2,D_ZX_1,D_YX_2] \right).
$$
Denoting 
$$
\Delta(u)=G_u(X_1,X_1)G_u(X_2,X_2)-G_u(X_1,X_2)^2,
$$
one can verify that the condition $\Delta(u)\neq 0$ is independent of the choice of the basis $u$. When this condition holds, we say that the surface in {\it non-degenerate}.
Along this paper, we shall always assume that the surface $M$ is non-degenerate. For a non-degenerate surface, define
$$
g(Y,Z)=\frac{1}{\Delta(u)^{1/3}}G_u(Y,Z).
$$
Then $g$ is independent of $u$ and is called the {\it affine metric} of the surface. 

Consider a $g$-orthonormal local frame $\{X_1,X_2\}$ of $M$. Any other $g$-orthonormal local frame $\{Y_1,Y_2\}$ is related to $\{X_1,X_2\}$ by 
\begin{equation}\label{eq:ChangeFrame1}
\begin{array}{c}
Y_1=\cos(\theta)X_1+\sin(\theta)X_2\\
Y_2=-\sin(\theta)X_1+\cos(\theta)X_2,
\end{array}
\end{equation}
for $\epsilon=1$ and 
\begin{equation}\label{eq:ChangeFrame2}
\begin{array}{c}
Y_1=\cosh(\theta)X_1+\sinh(\theta)X_2\\
Y_2=\sinh(\theta)X_1+\cosh(\theta)X_2,
\end{array}
\end{equation}
for $\epsilon=-1$, for some $\theta$.
It is verified in \cite{Nomizu93}, lemmas 4.1 and 4.2, that the corresponding local frame $\{\overline{\xi}_1,\overline{\xi}_2\}$ for $\sigma$ satisfying \eqref{eq:Normalxi} is given by
\begin{equation}
\begin{array}{c}\label{eq:Changexi1}
\overline{\xi}_1=\cos(2\theta)\xi_1+\sin(2\theta)\xi_2\\
\overline{\xi}_2=-\sin(2\theta)\xi_1+\cos(2\theta)\xi_2,
\end{array}
\end{equation}
for $\epsilon=1$ and 
\begin{equation}\label{eq:Changexi2}
\begin{array}{c}
\overline{\xi}_1=\cosh(2\theta)\xi_1+\sinh(2\theta)\xi_2\\
\overline{\xi}_2=\sinh(2\theta)\xi_1+\cosh(2\theta)\xi_2,
\end{array}
\end{equation}
for $\epsilon=-1$.

\subsection{Shape operators}

The shape operators $S_1$ and $S_2$ are defined by equation \eqref{eq:shape} and its components $\lambda_{ij}^k$ are defined by \eqref{eq:definelambda}.
In this section we show how the matrix $L$ defined by \eqref{eq:defineL} changes by a change of the $g$-orthonormal local frame $\{X_1,X_2\}$. 


In order to have a more compact notation, consider the matrices $R_{\epsilon}$, $\epsilon=\pm 1$,  given by
$$
R_{1}(\theta)=
\left[
\begin{array}{cc}
\cos(\theta) & \sin(\theta)\\
-\sin(\theta) & \cos(\theta)
\end{array}
\right];\ \ 
R_{-1}(\theta)=
\left[
\begin{array}{cc}
\cosh(\theta) & \sinh(\theta)\\
\sinh(\theta) & \cosh(\theta)
\end{array}
\right].
$$

\begin{lemma}\label{lemma:ChangeL}
Denote by $\overline{L}$ the matrix $L$ associated with the local frame $\{Y_1,Y_2\}$ defined by \eqref{eq:ChangeFrame1} and \eqref{eq:ChangeFrame2}. Then
\begin{equation}\label{eq:ChangeL}
\overline{L}=R_{\epsilon}(\theta)LR_{\epsilon}(3\epsilon\theta).
\end{equation}
\end{lemma}
\begin{proof}
The proof are long but straightforward calculations. For example, in case $\epsilon=-1$, we can calculate the first row of $\overline{L}$ as follows: From equation \eqref{eq:Changexi2} we have that
\begin{equation*}
\begin{array}{c}
\overline{S}_1(Y_1)=\cosh(\theta)\cosh(2\theta)S_1(X_1)+\cosh(\theta)\sinh(2\theta)S_2(X_1)+\\
+\sinh(\theta)\cosh(2\theta)S_1(X_2)+\sinh(\theta)\sinh(2\theta)S_2(X_2)
\end{array}
\end{equation*}
and 
\begin{equation*}
\begin{array}{c}
\overline{S}_2(Y_1)=\cosh(\theta)\sinh(2\theta)S_1(X_1)+\cosh(\theta)\cosh(2\theta)S_2(X_1)+\\
+\sinh(\theta)\sinh(2\theta)S_1(X_2)+\sinh(\theta)\cosh(2\theta)S_2(X_2)
\end{array}
\end{equation*}
Now using again equations \eqref{eq:ChangeFrame2} and comparing the coefficients we obtain after some calculations
\begin{equation*}
\begin{array}{c}
\overline{L}_{11}=\cosh(\theta)\cosh(3\theta)L_{11}-\cosh(\theta)\sinh(3\theta)L_{12}+\\
+\sinh(\theta)\cosh(3\theta)L_{21}-\sinh(\theta)\sinh(3\theta)L_{22}
\end{array}
\end{equation*}
and 
\begin{equation*}
\begin{array}{c}
\overline{L}_{12}=-\cosh(\theta)\sinh(3\theta)L_{11}+\cosh(\theta)\cosh(3\theta)L_{12}-\\
-\sinh(\theta)\sinh(3\theta)L_{21}+\sinh(\theta)\cosh(3\theta)L_{22},
\end{array}
\end{equation*}
which agree with equation \eqref{eq:ChangeL}.
\end{proof}

\subsection{Cubic forms}

Consider the cubic forms $C^1$ and $C^2$ defined by equation \eqref{eq:defineCubic} and the matrix $F$ whose entries are defined by equations \eqref{eq:defineF}.

\begin{lemma}\label{lemma:ChangeF}
Denote by $\overline{F}$ the matrix $F$ associated with the local frame $\{Y_1,Y_2\}$ defined by \eqref{eq:ChangeFrame1} and \eqref{eq:ChangeFrame2}. Then
\begin{equation}\label{eq:MudF}
\overline{F}=R_{\epsilon}(2\epsilon\theta)FR_{\epsilon}(3\epsilon\theta).
\end{equation}
\end{lemma}
\begin{proof}
We give a proof in case $\epsilon=1$, the case $\epsilon=-1$ being similar. Using complex notation, observe that
$$
C^1(X_1+iX_2, X_1+iX_2,X_1+iX_2)=F_{12}+iF_{11};
$$
$$
C^2(X_1+iX_2, X_1+iX_2,X_1+iX_2)=F_{22}+iF_{21};
$$
By lemma 6.2 of \cite{Nomizu93},
$$
e^{3i\theta}\overline{C}^1(Y_1+iY_2,Y_1+iY_2,Y_1+iY_2)=
$$
$$
 \cos(2\theta)C^1(X_1+iX_2, X_1+iX_2,X_1+iX_2)+\sin(2\theta)C^2(X_1+iX_2, X_1+iX_2,X_1+iX_2).
$$
$$
e^{3i\theta}\overline{C}^2(Y_1+iY_2,Y_1+iY_2,Y_1+iY_2)=
$$
$$
-\sin(2\theta)C^1(X_1+iX_2, X_1+iX_2,X_1+iX_2)+\cos(2\theta)C^2(X_1+iX_2, X_1+iX_2,X_1+iX_2).
$$
Thus
$$
\begin{array}{c}
\overline{F}_{12}+i\overline{F}_{11}=e^{-3i\theta} \left[ \cos(2\theta)(F_{12}+iF_{11})+\sin(2\theta)(F_{22}+iF_{21}) \right]\\
\overline{F}_{22}+i\overline{F}_{21}=e^{-3i\theta} \left[ -\sin(2\theta)(F_{12}+iF_{11})+\cos(2\theta)(F_{22}+iF_{21}) \right],
\end{array}
$$
which can be written as in equation \eqref{eq:MudF}. 
\end{proof}

Now we can prove the following lemma:

\begin{lemma}\label{lemma:rankH}
The rank of $H$ is independent of the choice of the local frame $\{X_1,X_2\}$. Moreover, if $rank(H)=1$, then $\overline{\eta}=\eta+3\theta$.
\end{lemma}
\begin{proof}
By lemmas \ref{lemma:ChangeL} and \ref{lemma:ChangeF}, the column-vector $[\overline{A},\overline{B}]^t$ belongs to the kernel of $\overline{H}$ if and only if $[A,B]^t$ belongs to the kernel of $H$, 
where $[\overline{A},\overline{B}]^t=R_{\epsilon}(-3\epsilon\theta)[A,B]^t$, which implies the invariance of $rank(H)$. In case $\epsilon=1$, we have 
that 
\begin{equation*}
\tan(\eta+3\theta)=\frac{\sin(\eta)\cos(3\theta)+\cos(\eta)\sin(3\theta)}{\cos(\eta)\cos(3\theta)-\sin(\eta)\sin(3\theta)}=\frac{B\cos(3\theta)+A\sin(3\theta)}{A\cos(3\theta)-B\sin(3\theta)}=\frac{\overline{B}}{\overline{A}},
\end{equation*}
thus proving that $\overline{\eta}=\eta+3\theta$. Similarly, in case $\epsilon=-1$, 
\begin{equation*}
\tanh(\eta+3\theta)=\frac{B\cosh(3\theta)+A\sinh(3\theta)}{A\cosh(3\theta)+B\sinh(3\theta)}=\frac{\overline{B}}{\overline{A}},
\end{equation*}
again proving that $\overline{\eta}=\eta+3\theta$.
\end{proof}

\subsection{Some formulas}

For further references, we write some formulas that hold for any transversal bundle $\sigma$. 
The symmetry conditions on the cubic forms imply that
\begin{equation}\label{eq:SymmetryC1}
\begin{array}{c}
2\Gamma_{22}^2+\tau_1^1(X_2)=-\Gamma_{12}^1+\epsilon\Gamma_{11}^2+\tau_2^1(X_1)\\
-2\epsilon\Gamma_{11}^1-\epsilon\tau_1^1(X_1)=\epsilon\Gamma_{21}^2-\Gamma_{22}^1+\tau_2^1(X_2)
\end{array}
\end{equation}
and
\begin{equation}\label{eq:SymmetryC2}
\begin{array}{c}
-2\Gamma_{12}^1-\epsilon\tau_1^2(X_1)=\tau_2^2(X_2)\\
-2\Gamma_{21}^2+\tau_1^2(X_2)=\tau_2^2(X_1)
\end{array}
\end{equation}

On the other hand, the condition $[X_1,X_2,\xi_1,\xi_2]=1$ implies that
\begin{equation}\label{eq:Det1}
\begin{array}{c}
\Gamma_{11}^1+\Gamma_{12}^2+\tau_1^1(X_1)+\tau_2^2(X_1)=0\\
\Gamma_{21}^1+\Gamma_{22}^2+\tau_1^1(X_2)+\tau_2^2(X_2)=0,
\end{array}
\end{equation}
(see \cite{Nomizu93}). 

\section{ The affine normal plane bundle}

\subsection{Definition and some relations}

Consider a $g$-ortonormal local frame $\{X_1,X_2\}$ of the tangent bundle. We say that a transversal bundle $\sigma$ is equiaffine if
\begin{equation*}
\begin{array}{c}
\epsilon\nabla(g)(X_1,X_1,X_1)+\nabla(g)(X_1,X_2,X_2)=0\\
\epsilon\nabla(g)(X_2,X_1,X_1)+\nabla(g)(X_2,X_2,X_2)=0
\end{array}
\end{equation*}
The affine normal plane bundle is an equiaffine bundle $\sigma$ satisfying 
\begin{equation*}
\begin{array}{c}
\nabla(g)(X_2,X_1,X_1)+\nabla(g)(X_1,X_2,X_1)=0\\
\nabla(g)(X_1,X_2,X_2)+\nabla(g)(X_2,X_1,X_2)=0
\end{array}
\end{equation*}
Lemma 7.3 of \cite{Nomizu93} says that the affine normal plane is characterized by the conditions 
\begin{equation}\label{eq:AffineBundle1}
\Gamma_{12}^2=-\Gamma_{11}^1;\ \Gamma_{21}^1=-\Gamma_{22}^2;
\end{equation}
and 
\begin{equation}\label{eq:AffineBundle2}
2\Gamma_{11}^1=\Gamma_{21}^2+\epsilon\Gamma_{22}^1;\ 2\Gamma_{22}^2=\Gamma_{12}^1+\epsilon\Gamma_{11}^2.
\end{equation}
As a consequence of equations \eqref{eq:Det1} and \eqref{eq:AffineBundle1} we obtain
\begin{equation}\label{eq:SumTau}
\tau_1^1+\tau_2^2=0.
\end{equation}

It is proved in \cite{Nomizu93} that a non-degenerate immersion admits a unique affine normal bundle. 


\subsection{Characterization of the affine normal bundle in terms of the cubic forms}

Define
\begin{equation*}\label{eq:E1E2}
\begin{array}{c}
E_1=\epsilon C^1(X_1,X_1,X_1)+C^1(X_1,X_2,X_2)-\epsilon C^2(X_1,X_1,X_2)-C^2(X_2,X_2,X_2)\\
E_2=\epsilon C^1(X_1,X_1,X_2)+C^1(X_2,X_2,X_2)+\epsilon C^2(X_1,X_2,X_2)+C^2(X_1,X_1,X_1)\\
E_3=3C^1(X_1,X_1,X_1)-\epsilon C^1(X_1,X_2,X_2)+3C^2(X_1,X_1,X_2)-\epsilon C^2(X_2,X_2,X_2)\\
E_4=C^1(X_1,X_1,X_2)-3\epsilon C^1(X_2,X_2,X_2)+3C^2(X_1,X_2,X_2)-\epsilon C^2(X_1,X_1,X_1).
\end{array}
\end{equation*}

\begin{Proposition}\label{prop:NormalCubic}
$\sigma$ is the affine normal plane bundle if and only if $E_1=E_2=E_3=E_4=0$. 
\end{Proposition}

\begin{proof}
For a general transversal bundle $\sigma$, the components of the cubic form are given by
\begin{equation}\label{eq:CubicFormulas}
\begin{array}{c}
C^1(X_1,X_1,X_1)=-2\Gamma_{11}^1+\tau_1^1(X_1),\\
C^1(X_1,X_1,X_2)=-2\Gamma_{21}^1+\tau_1^1(X_2),\\
C^1(X_1,X_2,X_2)=2\epsilon\Gamma_{12}^2-\epsilon\tau_1^1(X_1),\\
C^1(X_2,X_2,X_2)=2\epsilon\Gamma_{22}^2-\epsilon\tau_1^1(X_2),
\end{array}
\begin{array}{c}
C^2(X_1,X_1,X_1)=-2\Gamma_{11}^2+\tau_1^2(X_1),\\
C^2(X_2,X_1,X_1)=-2\Gamma_{21}^2+\tau_1^2(X_2),\\
C^2(X_1,X_2,X_2)=-2\Gamma_{12}^1-\epsilon\tau_1^2(X_1),\\
C^2(X_2,X_2,X_2)=-2\Gamma_{22}^1-\epsilon\tau_1^2(X_2).
\end{array}
\end{equation}
Assuming that $\sigma$ is the affine normal plane bundle, equations \eqref{eq:AffineBundle1} and \eqref{eq:AffineBundle2} easily imply that $E_1=E_2=0$. Moreover, it is not difficult
to verify that equations \eqref{eq:AffineBundle1} and \eqref{eq:AffineBundle2} together with equations \eqref{eq:SymmetryC1}, \eqref{eq:SymmetryC2} and \eqref{eq:Det1} imply that
$E_3=E_4=0$. 

Assume now that $E_1=E_2=E_3=E_4=0$. Then we can write
\begin{equation*}
\begin{array}{c}
-\Gamma_{11}^1+\Gamma_{12}^2+\Gamma_{21}^2+\epsilon\Gamma_{22}^1=0\\
-3\Gamma_{11}^1-\Gamma_{12}^2-3\Gamma_{21}^2+\epsilon\Gamma_{22}^1=-2(\tau_1^1(X_1)+\tau_1^2(X_2)).
\end{array}
\end{equation*}
and 
\begin{equation*}
\begin{array}{c}
-\Gamma_{21}^1+\Gamma_{22}^2-\epsilon\Gamma_{11}^2-\Gamma_{12}^1=0\\
\epsilon\Gamma_{11}^2-3\Gamma_{12}^1-\Gamma_{21}^1-3\Gamma_{22}^2=-2(\tau_1^1(X_2)-\epsilon\tau_1^2(X_1)).
\end{array}
\end{equation*}
By using equations \eqref{eq:SymmetryC2} we obtain
\begin{equation*}
\begin{array}{c}
-3\Gamma_{11}^1-\Gamma_{12}^2+\Gamma_{21}^2+\epsilon\Gamma_{22}^1=-2(\tau_1^1+\tau_2^2)(X_1)\\
\Gamma_{11}^1+\Gamma_{12}^2=(\tau_1^1+\tau_2^2)(X_1),
\end{array}
\end{equation*}
and
\begin{equation*}
\begin{array}{c}
-3\Gamma_{22}^2-\Gamma_{21}^1+\Gamma_{12}^1+\epsilon\Gamma_{11}^2=-2(\tau_1^1+\tau_2^2)(X_2)\\
\Gamma_{22}^2+\Gamma_{21}^1=(\tau_1^1+\tau_2^2)(X_2).
\end{array}
\end{equation*}
Now we use equations \eqref{eq:Det1} to conclude that equations \eqref{eq:AffineBundle1} and \eqref{eq:AffineBundle2} hold, which proves that $\sigma$ is the affine normal plane bundle.
\end{proof}

\begin{remark}
There is another choice of the transversal bundle $\sigma$ introduced by Klingenberg (\cite{Klingenberg51}) that is characterized by four conditions
involving the cubic forms $C^1$ and $C^2$ (see lemma 6.1. of \cite{Nomizu93}). Two of these conditions are $E_1=E_2=0$. 
\end{remark}

When we choose the affine normal bundle as the transversal bundle $\sigma$, the elements $F_{ij}$ of the matrix $F$ assume a remarkable simple form.

\begin{Proposition} For the affine normal plane bundle 
\begin{equation}\label{eq:FNormal}
F=\left[
\begin{array}{cc}
F_{11}& F_{12}\\
F_{21}& F_{22}
\end{array}
\right]=4\left[
\begin{array}{cc}
\Gamma_{22}^2+\tau_1^1(X_2)  & \epsilon\Gamma_{11}^1+\epsilon\tau_1^1(X_1) \\
\Gamma_{11}^1-\tau_1^1(X_1) & -\Gamma_{22}^2+\tau_1^1(X_2) 
\end{array}
\right]
\end{equation}
\end{Proposition}
\begin{proof}
We shall check these formulas for $F_{12}$, the other cases being similar. From equations \eqref{eq:CubicFormulas}, we have
$$
F_{12}=\epsilon C^1(X_1,X_1,X_1)-3C^1(X_1,X_2,X_2)=
$$
$$
=\epsilon \left[-2\Gamma_{11}^1+\tau_1^1(X_1) \right]-3\epsilon\left[ 2\Gamma_{12}^2 -\tau_1^1(X_1)    \right]=4\epsilon\left[ \Gamma_{11}^1+\tau_1^1(X_1)  \right],
$$
where in last equality we have used equations \eqref{eq:AffineBundle1}.
\end{proof}

\subsection{Invariance of equations \eqref{eq:DerivEta} under the choice of the local frame}

Consider $G_1$ and $G_2$ defined by equations \eqref{eq:defineG}. 

\begin{lemma}
When $\sigma$ is the affine normal plane bundle we can write
\begin{equation}
\begin{array}{c}
G_1=5\Gamma_{22}^2-3\epsilon\Gamma_{11}^2\\
G_2=5\Gamma_{11}^1-3\epsilon\Gamma_{22}^1.
\end{array}
\end{equation}
\end{lemma}
\begin{proof}
We shall prove the above formula for $G_1$, the proof for $G_2$ is similar.
We have 
$$
G_1=\Gamma_{22}^2-\epsilon\Gamma_{11}^2+\tau_1^1(X_2)-\epsilon\tau_1^2(X_1)=\Gamma_{22}^2-\epsilon\Gamma_{11}^2+2\Gamma_{12}^1,
$$
where we have used formulas \eqref{eq:SymmetryC2} and \eqref{eq:SumTau}. 
Now using equations \eqref{eq:AffineBundle2}, we obtain the desired formula. 
\end{proof}

\begin{lemma}
We have that
\begin{equation}
\left[
\begin{array}{c}
\overline{G}_1\\
\overline{G}_2
\end{array}
\right]
=R_{\epsilon}(-\epsilon\theta)
\left[
\begin{array}{c}
G_1\\
G_2
\end{array}
\right]
+3
\left[
\begin{array}{cc}
-\epsilon & 0 \\
0 & 1
\end{array}
\right]
R_{\epsilon}(\theta)
\left[
\begin{array}{c}
d\theta(X_1)\\
d\theta(X_2)
\end{array}
\right].
\end{equation}
\end{lemma}
\begin{proof}
We consider the case $\epsilon=1$, the case $\epsilon=-1$ being similar. From equations \eqref{eq:ChangeFrame1} we obtain
\begin{equation*}
\begin{array}{c}
\nabla_{Y_1}Y_1=\cos^2(\theta)\nabla_{X_1}X_1+\sin(\theta)\cos(\theta)(\nabla_{X_1}X_2+\nabla_{X_2}X_1)+\sin^2(\theta)\nabla_{X_2}X_2\\
              +\left[ d\theta(X_1)\cos(\theta)+d\theta(X_2)\sin(\theta) \right]Y_2
\end{array}
\end{equation*}
\begin{equation*}
\begin{array}{c}
\nabla_{Y_2}Y_2=\sin^2(\theta)\nabla_{X_1}X_1-\sin(\theta)\cos(\theta)(\nabla_{X_1}X_2+\nabla_{X_2}X_1)+\cos^2(\theta)\nabla_{X_2}X_2\\
              -\left[ -d\theta(X_1)\sin(\theta)+d\theta(X_2)\cos(\theta) \right]Y_1
\end{array}
\end{equation*}
Using again equations \eqref{eq:ChangeFrame1} we obtain
$$
\overline{G}_1=-3\left(\cos(\theta)d\theta(X_1)+\sin(\theta)d\theta(X_2)\right)+\cos^3(\theta)G_1-\sin^3(\theta)G_2+
$$
$$
+5\left[\sin^2(\theta)\cos(\theta)(\Gamma_{11}^2+\Gamma_{12}^1+\Gamma_{21}^1)-\sin(\theta)\cos^2(\theta)(\Gamma_{22}^1+\Gamma_{12}^2+\Gamma_{21}^2) \right]
$$
$$
-3\left[\sin^2(\theta)\cos(\theta)(\Gamma_{22}^2-\Gamma_{12}^1-\Gamma_{21}^1)+\sin(\theta)\cos^2(\theta)(\Gamma_{21}^2+\Gamma_{12}^2-\Gamma_{11}^1) \right].
$$
Using now equations \eqref{eq:AffineBundle1} and \eqref{eq:AffineBundle2} we obtain
\begin{equation*}
\overline{G}_1=\cos(\theta)G_1-\sin(\theta)G_2-3(\cos(\theta)d\theta(X_1)+\sin(\theta)d\theta(X_2))
\end{equation*}
Similar calculations leads to
\begin{equation*}
\overline{G}_2=\sin(\theta)G_1+\cos(\theta)G_2+3(-\sin(\theta)d\theta(X_1)+\cos(\theta)d\theta(X_2)),
\end{equation*}
thus proving the lemma. 
\end{proof}

\begin{corollary}
We have that
\begin{equation}
\left[
\begin{array}{c}
d\overline\eta(Y_1)+\epsilon \overline{G}_1\\
d\overline\eta(Y_2)-\overline{G}_2
\end{array}
\right]
=R_{\epsilon}(\theta)\cdot\left[
\begin{array}{c}
d\eta(X_1) +\epsilon G_1 \\
d\eta(X_2) -G_2
\end{array}
\right].
\end{equation}
\end{corollary}
\begin{proof}
By lemma \ref{lemma:rankH}, $\overline{\eta}=\eta+3\theta$. Thus, if $\epsilon=1$,
\begin{equation*}
\begin{array}{c}
d\overline{\eta}(Y_1)=\cos(\theta)d\eta(X_1)+\sin(\theta)d\eta(X_2)+3(\cos(\theta)d\theta(X_1)+\sin(\theta)d\theta(X_2))\\
d\overline{\eta}(Y_2)=-\sin(\theta)d\eta(X_1)+\cos(\theta)d\eta(X_2)+3(-\sin(\theta)d\theta(X_1)+\cos(\theta)d\theta(X_2)),
\end{array}
\end{equation*}
which implies that
\begin{equation*}
\begin{array}{c}
d\overline{\eta}(Y_1)+\overline{G}_1=\cos(\theta)(d\eta(X_1)+G_1)+\sin(\theta)(d\eta(X_2)-G_2)\\
d\overline{\eta}(Y_2)-\overline{G}_2=-\sin(\theta)(d\eta(X_1)+G_1)+\cos(\theta)(d\eta(X_2)-G_2).
\end{array}
\end{equation*}

If $\epsilon=-1$, 
\begin{equation*}
\begin{array}{c}
d\overline{\eta}(Y_1)=\cosh(\theta)d\eta(X_1)+\sinh(\theta)d\eta(X_2)+3(\cosh(\theta)d\theta(X_1)+\sinh(\theta)d\theta(X_2))\\
d\overline{\eta}(Y_2)=\sinh(\theta)d\eta(X_1)+\cosh(\theta)d\eta(X_2)+3(\sinh(\theta)d\theta(X_1)+\cosh(\theta)d\theta(X_2)),
\end{array}
\end{equation*}
implying that
\begin{equation*}
\begin{array}{c}
d\overline{\eta}(Y_1)-\overline{G}_1=\cosh(\theta)(d\eta(X_1)-G_1)+\sinh(\theta)(d\eta(X_2)-G_2)\\
d\overline{\eta}(Y_2)-\overline{G}_2=\sinh(\theta)(d\eta(X_1)-G_1)+\cosh(\theta)(d\eta(X_2)-G_2),
\end{array}
\end{equation*}
thus proving the corollary.
\end{proof}

\section{Proof of the Main Theorem}

We begin with the following lemma:
\begin{lemma}\label{lemma:Proof}
The system of equations 
\begin{equation}\label{eq:Equiv1}
dA(X_1)=G_1B;\ dA(X_2)=-\epsilon G_2B;\ dB(X_1)=-\epsilon G_1A;\ dB(X_2)=G_2A;
\end{equation}
is equivalent to 
\begin{equation}\label{eq:Equiv2}
A^2+\epsilon B^2=c;\ d\eta(X_1)=-\epsilon G_1;\ d\eta(X_2)=G_2,
\end{equation}
for some constant $c$, where $\tan(\eta)=\frac{B}{A}$, if $\epsilon=1$, and $\tanh(\eta)=\frac{B}{A}$, if $\epsilon=-1$.
\end{lemma}
\begin{proof}
If we assume that equations \eqref{eq:Equiv1} hold, then 
\begin{equation*}
AdA(X_1)+\epsilon BdB(X_1)=0; \ \ AdA(X_2)+\epsilon BdB(X_2)=0,
\end{equation*}
which implies $A^2+\epsilon B^2=c$, for some costant $c\neq 0$, and
\begin{equation*}
d\eta(X_1)=-\epsilon G_1;\ \ d\eta(X_2)=G_2.
\end{equation*}
On the other hand, if equations \eqref{eq:Equiv2} hold, then we can define 
\begin{equation*}
\tilde{G}_1=\frac{dA(X_1)}{B}=-\epsilon\frac{dB(X_1)}{A}
\end{equation*}
to obtain 
\begin{equation}
d\eta(X_1)=-\epsilon \tilde{G}_1
\end{equation}
and conclude that $\tilde{G}_1=G_1$. In a similar way we show that $dA(X_2)=-\epsilon G_2B;\ dB(X_2)=G_2A$,
which completes the proof of the lemma.
\end{proof}

\paragraph{Proof of theorem \ref{thm1}, part 1:}

Assume that $\Omega$ is a parallel symplectic form such that $S$ is $\Omega$-Lagrangian. Differentiating 
\begin{equation}
\Omega(X_1,X_2)=0
\end{equation}
with respect to $X_1$ and $X_2$ we obtain
\begin{equation*}
\begin{array}{c}
\Omega(D_{X_1}X_1,X_2)+\Omega(X_1,D_{X_1}X_2)=0\\
\Omega(D_{X_2}X_1,X_2)+\Omega(X_1,D_{X_2}X_2)=0,
\end{array}
\end{equation*}
which is equivalent to 
\begin{equation*}
\begin{array}{c}
\Omega(\xi_1,X_2)+\Omega(X_1,\xi_2)=0\\
\Omega(\xi_2,X_2)+\Omega(X_1,-\epsilon\xi_1)=0,
\end{array}
\end{equation*}
Write then
\begin{equation}
\Omega(X_1,\xi_2)=A,\ \Omega(X_2,\xi_1)=A; \ \ \Omega(X_1,\xi_1)=B,\ \Omega(X_2,\xi_2)=-\epsilon B; 
\end{equation}
for some functions $A$ and $B$. 

Differentiating $A$ with respect to $X_1$ in the first two equations we obtain
\begin{equation}\label{eq:DA1}
\begin{array}{c}
dA(X_1)=\left( \Gamma_{11}^1+\tau_2^2(X_1) \right) A+ \left( -\epsilon\Gamma_{11}^2 +\tau_2^1(X_1) \right) B+\Omega(\xi_1,\xi_2); \\
dA(X_1)= \left( \Gamma_{12}^2+\tau_1^1(X_1) \right)A +\left( \Gamma_{12}^1-\epsilon \tau_1^2(X_1) \right)B-\Omega(\xi_1,\xi_2).
\end{array}
\end{equation}
or equivalently
\begin{equation*}
\begin{array}{c}
(\Gamma_{11}^1+\tau_2^2(X_1)-\Gamma_{12}^2-\tau_1^1(X_1)) A +2\Omega(\xi_1,\xi_2)= (\Gamma_{12}^1-\epsilon\tau_1^2(X_1)+\epsilon\Gamma_{11}^2-\tau_2^1(X_1)) B\\
2dA(X_1)=(\Gamma_{11}^1+\Gamma_{12}^2+\tau_1^1(X_1)+\tau_2^2(X_1))A+(\Gamma_{12}^1-\epsilon\Gamma_{11}^2+\tau_2^1(X_1)-\epsilon\tau_1^2(X_1))B.
\end{array}
\end{equation*}

By using equations  \eqref{eq:SymmetryC1}, \eqref{eq:SymmetryC2}, \eqref{eq:Det1}, \eqref{eq:AffineBundle1} and \eqref{eq:AffineBundle2}, we verify that these equations are equivalent to 
\begin{equation}\label{eq:DA11}
F_{21}A+F_{22}B+4\Omega(\xi_1,\xi_2)=0
\end{equation}
and 
\begin{equation}\label{eq:DA12}
dA(X_1)=G_1 B.
\end{equation}

Differentiating $A$ with respect to $X_2$ we obtain
\begin{equation}\label{eq:DA2}
\begin{array}{c}
dA(X_2)=\left( \Gamma_{21}^1+\tau_2^2(X_2) \right) A+ \left( -\epsilon\Gamma_{21}^2 +\tau_2^1(X_2) \right) B; \\
dA(X_2)= \left( \Gamma_{22}^2+\tau_1^1(X_2) \right)A +\left( \Gamma_{22}^1-\epsilon \tau_1^2(X_2) \right)B,
\end{array}
\end{equation}
which are equivalent to
\begin{equation}\label{eq:DA21}
F_{11} A +F_{12} B=0
\end{equation}
and
\begin{equation}\label{eq:DA22}
dA(X_2)=-\epsilon G_2 B.
\end{equation}

Now differentiate $B$ with respect to $X_1$ to obtain
\begin{equation}\label{eq:DB1}
\begin{array}{c}
dB(X_1)=\left( \Gamma_{11}^2+\tau_1^2(X_1) \right) A+ \left( \Gamma_{11}^1 +\tau_1^1(X_1) \right) B; \\
-\epsilon dB(X_1)= \left( \Gamma_{12}^1+\tau_2^1(X_1) \right)A -\epsilon\left( \Gamma_{12}^2+ \tau_2^2(X_1) \right)B.
\end{array}
\end{equation}
We can verify that these equations are equivalent to  
\begin{equation}\label{eq:DB11}
F_{11} A + F_{12} B=0
\end{equation}
and
\begin{equation}\label{eq:DB12}
dB(X_1)=-\epsilon G_1 A.
\end{equation}
Differentiating $B$ with respect to $X_2$ we get
\begin{equation}\label{eq:DB2}
\begin{array}{c}
dB(X_2)=\left( \Gamma_{21}^2+\tau_1^2(X_2) \right) A+ \left( \Gamma_{21}^1 +\tau_1^1(X_2) \right) B-\Omega(\xi_1,\xi_2); \\
-\epsilon dB(X_2)= \left( \Gamma_{22}^1+\tau_2^1(X_2) \right)A -\epsilon\left( \Gamma_{22}^2+ \tau_2^2(X_2) \right)B-\epsilon\Omega(\xi_1,\xi_2),
\end{array}
\end{equation}
which are equivalent to 
\begin{equation}\label{eq:DB21}
F_{21} A +F_{22} B-4\Omega(\xi_1,\xi_2)=0
\end{equation}
and
\begin{equation}\label{eq:DB22}
dB(X_2)=G_2 A.
\end{equation}

From equations \eqref{eq:DA11} and \eqref{eq:DB21} we conclude that $\Omega(\xi_1,\xi_2)=0$. 
It follows that equations \eqref{eq:DA11}, \eqref{eq:DA21}, \eqref{eq:DB11} and \eqref{eq:DB21} are reduced to
\begin{equation*}\label{eq:KernelF}
\begin{array}{c}
F_{11}A+F_{12}B=0\\
F_{21}A+F_{22}B=0,
\end{array}
\end{equation*}
which says that $[A,B]^t$ belongs to the kernel of $F$.

Differentiating $\Omega(\xi_1,\xi_2)=0$ we obtain
\begin{equation}
\begin{array}{c}
\lambda_{11}^1A-\lambda_{11}^2\epsilon B-\lambda_{21}^1B-\lambda_{21}^2A=0\\
\lambda_{12}^1A-\lambda_{12}^2\epsilon B-\lambda_{22}^1B-\lambda_{22}^2A=0,
\end{array}
\end{equation}
which can be written as 
\begin{equation*}\label{eq:KernelL}
\begin{array}{c}
L_{11}A+L_{12}B=0\\
L_{21}A+L_{22}B=0.
\end{array}
\end{equation*}
We conclude that $[A,B]^t$ belongs to the kernel of $L$ and hence $rank(H)<2$. 

Finally, by lemma \ref{lemma:Proof}, equations \eqref{eq:DA12}, \eqref{eq:DA22}, \eqref{eq:DB12} and \eqref{eq:DB22} are equivalent to 
$A^2+\epsilon B^2=c$, for some constant $c$ and to equations \eqref{eq:DerivEta}. Equation $A^2+\epsilon B^2=c$ implies $\Omega\wedge\Omega=c[\cdot,\cdot,\cdot,\cdot]$.

\bigskip\bigskip

\paragraph{Proof of theorem \ref{thm1}, part 2:}

Assume that $rank(H)=1$ and equations \eqref{eq:DerivEta} hold. Denote by
$[A,B]^t$ a column-vector in $Ker(H)$ satisfying $A^2+\epsilon B^2=c$, for some constant $c\neq 0$. Define the symplectic form $\Omega$ by the conditions
\begin{equation*}
\begin{array}{c}
\Omega(X_1,X_2)=\Omega(\xi_1,\xi_2)=0\\
\Omega(X_1,\xi_2)=\Omega(X_2,\xi_1)=A\\
\Omega(X_1,\xi_1)=-\epsilon \Omega(X_2,\xi_2)=B
\end{array}
\end{equation*}
We shall prove that the symplectic form $\Omega$ is parallel.

Observe first that
\begin{equation*}
\begin{array}{c}
D_{X_1}\Omega(X_1,X_2)=-A+A=0\\
D_{X_2}\Omega(X_1,X_2)=-B+ B=0.
\end{array}
\end{equation*}
Moreover
\begin{equation*}
\begin{array}{c}
D_{X_1}\Omega(\xi_1,\xi_2)=L_{11}A+L_{12}B=0\\
D_{X_2}\Omega(\xi_1,\xi_2)=L_{21}A+L_{22}B=0.
\end{array}
\end{equation*}
We must prove now that $(D_{X_k}\Omega)(X_i,\xi_j)=0$, for any $i,j,k=1,2$. We shall prove for $(i,j,k)=(1,2,1)$ and $(i,j,k)=(2,1,1)$, the other cases being similar. 
We have
\begin{equation*}
\begin{array}{c}
dA(X_1)-\Gamma_{11}^1A-\Gamma_{11}^2(-\epsilon B)-\tau_2^1(X_1)B-\tau_2^2(X_1)A=0\\
dA(X_1)-\Gamma_{12}^1B-\Gamma_{12}^2A-\tau_1^1(X_1)A-\tau_1^2(X_1)(-\epsilon B)=0.
\end{array}
\end{equation*}
But, as we have seen above, this pair of equations are equivalent to
\begin{equation*}
\begin{array}{c}
F_{21}A+F_{22}B=0\\
dA(X_1)-G_1 B=0,
\end{array}
\end{equation*}
which holds by lemma \ref{lemma:Proof}.

\end{document}